\documentclass[11pt]{article} 
\usepackage[margin=1in]{geometry}
\usepackage{lmodern}
\usepackage[utf8]{inputenc}
\usepackage{microtype}
\usepackage{hyperref}
\usepackage{subfiles}

\usepackage{amsmath,amssymb,amsthm,mathtools}
\usepackage{tikz-cd}

\theoremstyle{plain}
\newtheorem{theorem}{Theorem}[section]
\newtheorem{proposition}[theorem]{Proposition}
\newtheorem{lemma}[theorem]{Lemma}
\newtheorem{corollary}[theorem]{Corollary}

\theoremstyle{definition}
\newtheorem{definition}[theorem]{Definition}

\newtheorem{remark}[theorem]{Remark}
\newtheorem{question}[theorem]{Question}

\usepackage[english]{babel}

\newcommand{\N}{\mathbb{N}}
\newcommand{\Z}{\mathbb{Z}}
\newcommand{\Q}{\mathbb{Q}}
\newcommand{\R}{\mathbb{R}}

\renewcommand{\epsilon}{\varepsilon}
\newcommand{\act}{\curvearrowright}

\DeclareMathOperator{\LO}{LO}
\DeclareMathOperator{\Ar}{Ar}
\DeclareMathOperator{\Aut}{Aut}
\DeclareMathOperator{\SL}{SL}
\DeclareMathOperator{\GL}{GL}

\DeclareMathOperator{\LI}{LI}
\DeclareMathOperator{\id}{id}

\usepackage{xcolor}

\title{Borel Complexity of the Isomorphism Relation of Archimedean Orders in Finitely Generated Groups}
\author{Antoine Poulin}

\begin{document}
\maketitle
\abstract{In 2020, Calderoni, Marker, Motto Ros and Shani asked what the Borel complexity of the isomorphism relation of Archimedean orders on $\Q^n$ is. We answer this question by proving that the isomorphism relation of Archimedean orders on $\Z^n$ is not hyperfinite when $n \geq 3$ and not treeable when $n \geq 4$. As a corollary, we get that the isomorphism relation of Archimedean orders on $\Q^n$ is not hyperfinite when $n \geq 3$ and not treeable when $n \geq 4$.

\section{Introduction}
The study of invariant orders on groups has connections to algebra, dynamics and low-dimensional topology. For example, every countable group admitting a left-invariant order (called a left-order) admits an action on the real line $\R$ by orientation-preserving homeomorphisms. Different types of orders can be considered, such as Conradian orders, the existence of which is equivalent to local indicability, a purely algebraic condition. Recently, tools from descriptive set theory, namely from Borel complexity and reducibility, have been successfully applied to the space of left-orders on a group $G$, along with actions by $\Aut(G)$ and its restriction to an action of $G$ by conjugacy. A (slightly reformulated) question of Deroin, Navas, and Rivas asks whether the relation of conjugacy of left-orders is always smooth, i.e that it admits complete invariants. This was answered in the negative by Filippo Calderoni and Adam Clay in \cite{calderoni_borel_2023}, who showed that free groups have universal conjugacy relations on their space of left-orders. Further, they showed that if this relation is smooth, the group must admit a Conradian order. They expanded upon their techniques in \cite{calderoni_borel_2023}, where a connection to the L-space conjecture of low-dimensional topology is presented. In their most recent work \cite{calderoni_condensation_2023}, they show that the conjugacy relation of left-orders on the Baumslag--Solitar group $BS(1,2)$ is hyperfinite non-smooth, the first example of a finitely generated group for which this relation is not smooth or universal. In a paper by Calderoni, Marker, Motto Ros and Shani \cite{calderoni2023anticlassification}, it is shown that the isomorphism relation of Archimedean orders on $\Q^2$ is not smooth. It follows from their aguments that the isomorphism relation of Archimedean orders on $\Z^2$ is hyperfinite (See \cite{poulin}). In \cite[Question 2.9]{calderoni2023anticlassification}, the authors ask what the Borel complexity of the isomorphism relation of Archimedean orders on $\Q^n$ is. We answer this question by proving the following theorems about $\Z^n$:

\begin{theorem}\label{intro:thm1}
The isomorphism relation of Archimedean orders on $\Z^n$ is not hyperfinite if $n \geq 3$.
\end{theorem}

\begin{theorem}\label{intro:thm2}
The isomorphism relation of Archimedean orders on $\Z^n$ is not treeable if $n \geq 4$.
\end{theorem}

In particular, since the isomorphism relation of Archimedean orders on $\Z^n$ can be realized as a sub-equivalence relation of the  isomorphism relation of Archimedean orders on $\Q^n$, we get the following corollary, which answers \cite[Question 2.9]{calderoni2023anticlassification}:  

\begin{corollary}\label{cor-intro}
The isomorphism relation of Archimedean orders on $\Q^n$ is not hyperfinite if $n \geq 3$ and not treeable if $n \geq 4$.
\end{corollary}

In unpublished notes, Filippo Calderoni and Adam Clay showed that the isomorphism relation of Archimedean orders on $\Q^n$ is not smooth for $n \geq 3$. Corollary \ref{cor-intro} gives a significant strenghtening of this.

The strategy throughout is to tie the complexity of isomorphism relation of Archimedean orders on $\Z^n$ to that of $\SL_n(\Z)$ acting on $\R^n$. This is not done solely through reductions, but also through the use of class-preserving bijections. Here, the Archimedean condition translates to studying the action on a co-null set. We then apply measure-theoretic results by Zimmer \cite{zimmer_induced_1978} and Popa--Vaes \cite{popa_cocycle_2008}. In the proof of Theorem \ref{intro:thm1},  since the equivalence relation is an essentially free action of $\SL_n(\Z)$, this allows us to translate hyperfiniteness into Zimmer's notion of amenable actions. We then use that $\SL_n(\Z)$ is a lattice of $\SL_n(\R)$ see that amenability of the two actions are equivalent. Then, amenability of the (essentially-transitive) action of $\SL_n(R)$ is equivalent to amenability of the stabilizer group. Thus, the complexity of the essentially free action of $\SL_n(\Z)$ is still dictated (even almost everywhere) by the stabilizer group of a null set.

An immediate corollary of H\"{o}lder's theorem on Archimedean ordered groups (see Theorem \ref{thm:Holder}) implies that all finitely generated groups admitting Archimedean orders are isomorphic to $\Z^n$ for some $n$. Thus, Theorem \ref{intro:thm1} and \ref{intro:thm2} can be interpreted as a statement about all finitely generated Archimedean groups.

\subsection*{Acknowledgements}

We thank Marcin Sabok and Anush Tserunyan for their guidance. We thank Bachir Bekka, Filippo Calderoni, Damien Gaboriau and Samuel Mellick for helpful discussions and direction. We thank Filippo Calderoni for comments.

\section{Background and results}
\subsection{Orders on groups}
\subsubsection*{Left-orders on groups}
A group $G$ is \textbf{left-orderable} if there exists a total linear order $<$ on $G$ such that
\[gh < gk \Longleftrightarrow h < k.\]
In this case, $<$ is called a \textbf{left-order}. A \textbf{positive cone} $P$ is a subset $P \subset G$ such that:
\begin{itemize}
\item $P$ is closed under multiplication: $P \cdot P = P$.
\item $G$ splits as a disjoint union of $P$, $\{\id\}$ and $P^{-1} := \left\{g^{-1} : g \in P \right\}$.
\end{itemize}
There is a one-to-one correspondence between left-orders and positive cones, given by
\begin{align*}
\id < g &\Longrightarrow g \in P(<),\\
g^{-1}h \in P &\Longrightarrow g <_P h.
\end{align*}
We freely associate between a pair of associated $<$ and $P(<)$. This correspondance allows us to define the \textbf{space of left-orders} on $G$ as 
\[\LO(G) := \left\{P \subset G: P \text{ is a positive cone} \right\}.\]
This is a closed subset of $2^{G}$, hence a compact subspace with the product topology. If $G$ is countable, then $\LO(G)$ is separable.

\subsubsection*{Archimedean order on groups}

An order $P(<) \in \LO(G)$ is \textbf{Archimedean} if for all $g,h > \id$, there is $n \in \N$ such that $g < h^n$. Informally, if $h^n < g$ for all $n$, we say that $h$ is infinitely small with respect to $g$. The \textbf{space of Archimedean orders} is 
\[\Ar(G) := \left\{P \in \LO(G): P \text{ is Archimedean} \right\}.\]
This is a $G_\delta$ subset of $\LO(G)$, that is a countable intersection of open sets.

\begin{theorem}[Hölder's theorem, Theorem 2.6 in \cite{clay_ordered_2016}]\label{thm:Holder}
Every Archimedean ordered group is isomorphic as an ordered group to a subgroup of $\R$.
\end{theorem}

\subsubsection*{Actions on Orders}

\begin{definition}
Let $(G, <)$ be a left-ordered group and $\phi \in \Aut(G)$. There is a left-order $\phi \cdot <$ such that 
\[g (\phi \cdot <) h \Longleftrightarrow \phi^{-1}(g) < \phi^{-1}(h) \]
This yields an action $\Aut(G) \curvearrowright \LO(G)$. If $<$ is Archimedean, then so is $\phi \cdot <$. Thus, we have an action $\Aut(G) \curvearrowright \Ar(G)$.
\end{definition}

\subsection{Linear actions and contructions}
\subsubsection*{Linear actions on Euclidean space and on its dual}

We denote the dual of $\R^n$, that is the space of linear functionals on $\R^n$, by
\[\left (\R^n\right )^\ast = \left\{ f: \R^n \rightarrow \R \mathbin{|} f \text{ is linear} \right\}. \] 
There is an induced action $\GL_n(\Z) \act \left (\R^n\right )^\ast$ by $T\cdot f = f \circ T^{-1}$. There is a linear isomorphism $\widehat{\cdot}: \R^n \rightarrow \left (\R^n\right )^\ast$ given by $\widehat{x} = \langle x, \cdot \rangle$, where $\langle \cdot, \cdot \rangle$ denotes the inner product with respect to the canonical basis. 

\begin{lemma}\label{prop:bireddual}
For $x,y \in \R^n$, 
\[\exists T \in \GL_n(\Z), Tx = y \Longleftrightarrow \exists S \in \GL_n(\Z), S \widehat{x} = \widehat{y}.\]
\end{lemma}

\begin{proof}
The induced action $\GL_n(\R) \act \left (\R^n\right )^\ast$ is given by 
\begin{align*}
T \cdot \widehat{x} &= \widehat{x} \circ T^{-1} \\
&= \langle x, T^{-1} \cdot \rangle \\
&= \langle  \left (T^{-1}\right)^\ast x, \cdot \rangle\\
&= \widehat{\left (T^{-1}\right)^\ast x}
\end{align*}
where $\left (T^{-1}\right)^\ast$ is the adjoint of the inverse of $T$. Noting that $\GL_n(\Z)$ is closed under inverses and adjoints, we reach the proposition, since
\begin{align*}
Tx = y &\Longrightarrow \left(T^{\ast}\right)^{-1} \widehat{x} = \widehat{y}, \\ 
T \widehat{x} = \widehat{y} &\Longrightarrow \left (T^{-1}\right)^\ast x  = y.
\end{align*}
\end{proof}

\subsubsection*{Totally irrational vectors}
We say that $f \in \left (\R^n\right )^\ast$ is \textbf{totally irrational} if it is injective on $\Z^n$. Denote
\begin{align*}
\left (\R^n\right )^\ast_\text{ti} &= \left\{f \in \left (\R^n\right )^\ast: f \text{ is totally irrational} \right\} \\
\R^n_\text{ti} &= \left\{x \in \R^n: \widehat{x} \in \left (\R^n\right )^\ast_\text{ti}\right\}.
\end{align*}
In the sense of Lebesgue measure, most elements are totally irrational:
\begin{lemma}\label{lemma:fullmeas}
$\R^n_\text{ti}$ has full Lebesgue measure in $\R^n$.
\end{lemma}
\begin{proof}
If $x \not \in \R^n_\text{ti}$, then $\langle x, v \rangle = 0$ for some $v \in \Z^n$, i.e $ x \in v^\bot$. For any $v\in \Z^n$, $v^\bot$ is a codimension one subspace, hence has null measure. Thus,
\[\R^n -  \R^n_\text{ti} = \bigcup_{v \in \Z^n} v^\bot\]
is a countable union of null set. This shows the proposition.
\end{proof}

For functions $0 \neq f \in \left(\R^n\right )^\ast_\text{ti}$, define $[f]_{\text{sc}} = \left\{g \in \left(\R^n\right )^\ast_\text{ti} : \exists r > 0, f = gr\right\}$. Denote \[\left(\R^n\right )^\ast_\text{ti} / \R_{>0} = \left\{[f]_{\text{sc}} : f \neq 0 \right\}\] for the space of equivalence classes. Since scalar multiplication commutes with matrices, the action $\GL_n(\Z) \rightarrow \left(\R^n\right )^\ast_\text{ti}$ factors to an action $\GL_n(\Z) \rightarrow \left(R^n\right )^\ast_\text{ti}/ \R_{>0}$. Up to scalars, totally irrational functionals encode the dynamics of Archimedean orders, in the sense of the following proposition:

\begin{proposition}\label{prop:ArOrderStructure}
There is a $\GL_n(\Z)$-equivariant bijection between $\left (\R^n\right )^\ast_\text{ti} / \R_{>0}$ and  $\Ar(\Z^n)$ defined by
\[[f]_{\text{sc}} \mapsto P_f := \left\{x \in \Z^n : f(x) > 0 \right\}.\]
By equivariance, we mean that for all $T \in \GL_n(\Z)$,
\[P_{T \cdot f} = T \cdot P_f.\]
\end{proposition}
\begin{proof}
Notice that the $P_f$ does not depend on the choice of representative. By Hölder's theorem (Theorem \ref{thm:Holder}), for any $P \in\Ar(\Z^n)$, there is an order-preserving injective group homomorphism $\phi: \Z^n \rightarrow \R$. Notice that $x \in P \Longleftrightarrow \phi(x) > 0$ for each $x \in \Z^n$. We can extend $\phi$ linearly to some element $f \in \left (\R^n\right )^\ast_\text{ti}$ satisfying
\[\forall x \in \Z^n, x \in P \Longleftrightarrow f(x) > 0.\]
This shows the map is surjective. To see $P_f$ is Archimedean, notice that if $n f(x) > f(y)$, then $f(nx) > f(y)$, hence $nx - y \in P_f$. To show the map is injective, notice that if $f,f'$ are not positive multiples of another,there must be an open $U \subset \R^n$ such that $f|_U >0$ and $f'|_U < 0$ there must be some $x \in \Q^n \cap U$. If $k \in \N_{n > 0}$ such that $kx \in \Z^n$, then $kx \in P_f$, but $kx \not\in P_{f'}$. As for equivariance, suppose $P$ is an Archimedean order and $[f]_{sc} \in \left (\R^n\right )^\ast_\text{ti} / \R_{>0}$ satisfies
\[x \in P \Longleftrightarrow f(x) > 0.\]
For all $T \in \GL_n(\Z)$, 
\[x \in T P \Longleftrightarrow T^{-1} x \in P \Longleftrightarrow f(T^{-1}(x)) > 0 \Longleftrightarrow \left(T \cdot f\right)(x) > 0.\]
This shows equivariance.
\end{proof}

\subsection{Borel complexity}
\subsubsection*{Countable Borel equivalence relations and reductions}
A topology is \textbf{Polish} if it is separable and completely metrizable. A $G_\delta$ subset of a Polish space is itself Polish. A \textbf{standard Borel space} is a measurable space $(X, \mathcal{B})$ such that $\mathcal{B}$ is the Borel $\sigma$-algebra of some Polish topology on $X$. If $G$ is a countable group, $\LO(G)$ and $\Ar(G)$ are both standard Borel, being $G_\delta$ subsets of the Polish space $2^G$. If $X$ is standard Borel, so is $X^2$ with the product $\sigma$-algebra. A \textbf{Borel equivalence relation} is an equivalence relation $E \subset X^2$ which is Borel as a subset of $X^2$. We present Borel equivalence relations as $(X,E)$ or simply $E$ when $X$ is implicit. A Borel equivalence relation is \textbf{countable} if every equivalence class is countable. If $G$ is a Polish group and $G \act X$ is a Borel action of $G$ on some standard Borel space, we denote by $E(G\act X)$ the \textbf{orbit equivalence relation}:
\[x \mathbin{E(G \act X)} x' \Longleftrightarrow \exists g\in G, gx = x'.\]
This is not in general Borel, however it is analytic, hence measurable. If $G$ is a countable group, this is Borel. In fact, by a theorem of Feldman and Moore, all countable Borel equivalence relations arise as orbit equivalence relation of some countable group. If $(X,E)$ and $(Y, F)$ are two Borel equivalence relation, a \textbf{Borel reduction} from $E$ to $F$ is a Borel function $f: X \rightarrow Y$ such that 
\[x \mathbin{E} x' \longleftrightarrow f(x) \mathbin{F} f(x').\]
In this case, we denote $E \leq_B F$ and say $E$ is \textbf{Borel reducible} to $F$. Informally, this means $E$ is simpler than $F$. If $E \leq_B F$ and $F \leq_B E$, we write $E \sim_B F$ and say that $E$ and $F$ are \textbf{Borel bireducible}. Here, we never consider measurable or Baire-measurable reductions, hence we will omit ``Borel". The \textbf{$E$-saturation} of a set $A$ is the smallest $E$-invariant set containing $A$. It is denoted $[A]_E$. A \textbf{complete section} of $E$ is a Borel set $A \subset X$ which meets every $E$-class. Equivalently, a Borel set $A$ is a complete section if $[A]_E = X$. For any set $A\subset X$, we denote by $E|_A$ the restriction of $E$ to $A$.
\begin{proposition}[{\cite[Prop 2.6]{dougherty_structure_1994}}]\label{prop:section}
If $(X,E)$ is a countable Borel equivalence relation and $A \subset X$ is a complete section of $E$, we have $E \sim_B E|_A$.
\end{proposition}

\subsubsection*{Homomorphism and class-bijections}
If $(X,E)$ and $(Y,F)$ are two Borel equivalence relation, a \textbf{homomorphism} from $E$ to $F$ is a Borel function $f:X \rightarrow Y$ such that 
\[x \mathbin{E} x' \Longrightarrow f(x) \mathbin{F} f(x').\]
A homomorphism $f$ from $E$ to $F$ is \textbf{class-bijective} if its restriction to every $E$-class is a bijection.

\subsubsection*{Hyperfiniteness and treeability}

Let $(X, E)$ be a countable Borel equivalence relation. A graph $G \subset X^2$ is a \textbf{graphing} of $E$ if it is Borel as a subset of $X^2$ and its connected component are exactly the $E$-classes. A countable equivalence relation is \textbf{hyperfinite} if it admits a graphing each of whose connected component is a line. It is \textbf{treeable} if it admits a graphing each of whose connected component is a tree.

\begin{proposition}[{\cite[Prop. 1.3, 3.3]{jackson_countable_2002}}]\label{prop:complexclassJKL}
Let $E,F$ be two countable Borel equivalence relations.
\begin{enumerate}
\item[(a)] If $E$ is hyperfinite (resp. treeable) and $F \leq_B E$, then $F$ is also hyperfinite (resp. treeable).
\item[(b)] If $E$ is hyperfinite (resp. treeable) and $F \subset E$ is a subequivalence relation, then $F$ is also hyperfinite (resp. treeable).
\end{enumerate}
\end{proposition}
\begin{proposition}[{\cite[Cor. 4.12]{chen_structurable_2018}}]\label{prop:complexclassCK}
If $F$ is hyperfinite (resp. treeable) and $f$ is a class-bijective homomorphism from $E$ to $F$, then $E$ is also hyperfinite (resp. treeable).
\end{proposition}

\subsection{Properties of actions and equivalence relations}

\subsubsection*{Zimmer amenability}

Throughout this section, we deal with locally compact second countable group $G$, which we abbreviate to \textbf{lcsc}. $G \act X$ will always denote a measure preserving action of $G$ on a standard $\sigma$-finite measure space. Consider a separable Banach space $B$ and a continuous representation $\pi: G \rightarrow \LI(B)$ of $G$ into the linear isometries of $B$. This induces a dual representation $\pi^\ast: G \rightarrow \LI(B^\ast)$. One of several equivalent definitions for lcsc group states that an lcsc group $G$ is \textbf{amenable} if for each separable Banach space $B$, for each representations $\pi: G \rightarrow \LI(B)$, for each weakly compact, convex  $\pi^\ast$-invariant $K \subset B^\ast$, there is some  $\pi^\ast$-invariant $v \in K$. In \cite{zimmer_amenable_1978}, Zimmer introduces a notion of amenability for ergodic, $\sigma$-finite measure preserving actions $G \act X$ of lcsc groups. Given a separable Banach space $B$, we define a \textbf{Borel a.e cocycle} as a Borel map $\alpha: G \times X \rightarrow \LI(B)$ (The group of linear isometries is Polish, hence is standard Borel.) satisfying for all $g,h \in G$, for almost every $ x\in X$, $\alpha(h, gx)\alpha(g,x) = \alpha(hg,x)$. Given a Borel a.e cocycle $\alpha$, we define an \textbf{$\alpha$-invariant field} as a Borel map $K:X \rightarrow K(B^\ast)$ to the set of weakly compact subset of $B^\ast$ (The set of weakly compact subsets of $B^\ast$ admits a standard Borel structure) such that $K(x)$ is convex for all $x$ and for all $g\in G$, for almost every $x\in X$, $\alpha(g,x)K(x) = K(gx)$. Given an $\alpha$-invariant field $K$, an \textbf{invariant section} is a Borel map $v: X \rightarrow  B^\ast$ such that $v(x) \in K(x)$ and for all $g\in G$, for almost every $x\in X$, $\alpha(g,x)v(x) = v(gx)$.
\begin{definition}[\cite{zimmer_amenable_1978}]
Let $G$ be a lcsc group and $G\act X$ be an ergodic, $\sigma$-finite measure preserving action. This action is said to be \textbf{amenable} if for every separable Banach space $B$, for every Borel a.e cocycle $\alpha: G \times X \rightarrow \LI(B)$, for every $\alpha$-invariant field $K$, there exists an invariant section $v$.
\end{definition}
There is also a notion for ergodic $\sigma$-finite measure preserving countable Borel equivalence relations, introduced in \cite{zimmer_hyperfinite_1977}. We omit it since the definition is similar, reformulating the definition of Borel a.e cocycle, $\alpha$-invariant field and invariant section for equivalence relations. To state important properties, we will need the following definition: An action $G \act X$ is \textbf{essentially free} if it is free on a conull $G$-invariant set. It is \textbf{essentially transitive} if there is a conull orbit. A \textbf{lattice} in a locally compact group $G$ is a discrete subgroup $H$ such that the quotient $G /H$ admits a $G$-invariant Borel probability measure. The relevant properties are the following:
\begin{proposition}[{\cite[Prop 3.1, 3.4]{zimmer_induced_1978}}]
Let $G$ be a locally compact separable group, $G \act X$ is a $\sigma$-finite measure preserving action.
\begin{enumerate}
\item[(a)] If $G \act X$ is essentially transitive, then $G \act X$ is amenable if and only if the stabilizer of a point is almost surely amenable.
\item[(b)] Suppose $G \act X$ is an ergodic amenable action and $H \subset G$ is a lattice. If $H \act X$ also acts ergodically, then $H \act X$ is amenable.
\end{enumerate}
\end{proposition}
\begin{proposition}[{\cite[Prop. 3.2]{zimmer_hyperfinite_1977}}]\label{prop:freeactamen}
Let $G\act X$ be an essentially free, $\sigma$-finite measure preserving action of a countable discrete group $G$. The action $G\act X$ is amenable if and only if the orbit equivalence relation $E(G \act X)$ is amenable.
\end{proposition}
\begin{remark}
While {\cite[Prop. 3.2]{zimmer_hyperfinite_1977}} is actually for quasi-invariant probability measures, it is straightforward to turn an invariant $\sigma$-finite measure into an equivalent quasi-invariant probability measure.
\end{remark}

The following fact is noted in \cite{zimmer_orbit_1981}.
\begin{proposition}\label{prop:notamenact}
If $n\geq 3$, the action $\SL_n(\Z) \act \R^n$ is not amenable.
\end{proposition}
\begin{proof}
Because the action $\SL_n(\R) \act \R^n$ is transitive, we may turn to the stabilizer of any point. The stabilizer of $(1, 0, ..., 0, 0)$ contains a copy of $\SL_2(\R)$ as a closed subgroup, consisting of transformations acting only on the last two coordinates. Since $\SL_2(\R)$ is not amenable, containing a discrete copy of a free non-abelian group, the stabilizer is itself not amenable. Hence the action $\SL_n(\R) \act \R^n$ is not amenable. Since $\SL_n(\Z)$ is a lattice in $\SL_n(\R)$, it suffices to see that $\SL_n(\Z) \act \R^n$ is ergodic. This is directly implied by Moore's ergodicity theorem, see  \cite[Example 2.2.9]{zimmer_ergodic_1984}.
\end{proof}

\begin{theorem}[\cite{connes_amenable_1981}]
A countable equivalence relation is amenable iff it is hyperfinite almost everywhere.
\end{theorem}

\begin{corollary}\label{prop:nonamen3}
If $n \geq 3$, the equivalence relation $E(\SL_n(\Z) \act \R^n_{\text{ti}})$ is not hyperfinite.
\end{corollary}
\begin{proof}
Suppose to the contrary that the equivalence relation $E(\SL_n(\Z) \act \R^n_{\text{ti}})$ is hyperfinite. Then, by Lemma \ref{lemma:fullmeas}, $E(\SL_n(\Z) \act \R^n)$ is hyperfinite almost everywhere, hence amenable. Since this action is essentially free, Proposition \ref{prop:freeactamen} implies that the action $\SL_n(\Z) \act \R^n$ is amenable, contradicting Proposition \ref{prop:notamenact}.
\end{proof}

\begin{remark}
The definition of amenability used in \cite{connes_amenable_1981} is not the same as Zimmer's definition, however the two are shown to be equivalent in \cite{adams_amenability_1991}.
\end{remark}

\subsubsection*{Property (T) for equivalence relations}
The definition of property $(T)$ for equivalence relation is due to Zimmer in \cite{zimmer_cohomology_1981}. Throughout this section, $E$ is a countable Borel equivalence relation on a standard Borel space $X$ preserving a probability measure $\mu$. A countable Borel equivalence relation $E$ is \textbf{probability measure preserving} if the following equality holds for all $A \subset E$.
\[\int_X \#\{y : (x,y) \in A \} d\mu(x) = \int_X \#\{x : (x,y) \in A \} d\mu(y).\]
In this case, let $\nu(A) = \int_X \#\{y : (x,y) \in A \}) d\mu(x)$. This is a $\sigma$-finite measure. The fact that $\nu$ is $\sigma$-finite can be shown using the Feldman--Moore theorem. One also defines $E^{(2)} = \{(x,y,z) \in X^3: (x,y), (y,z) \in E\}$ and a measure $\nu^{(2)}$ on $E^{(2)}$ defined similarly (see \cite{popa_cocycle_2008}). Let $\mathcal{H}$ be a Hilbert space. In this context, a \textbf{Borel a.e cocycle} is a Borel map $\alpha: E \rightarrow \mathcal{U}(\mathcal{H})$ to the unitary operators on $\mathcal{H}$ such that for almost every $(x,y, z) \in E^{(2)}$, $\alpha(x,y)\alpha(y,z) = \alpha(x,z)$. A cocycle $\alpha$ has \textbf{almost-invariant sections} if there exists a sequence of Borel maps $v_n: X \rightarrow \mathcal{H}$ such that for almost every  $(x,y) \in E$, we have $\lim_n\left(v_n(x) - \alpha(x,y)v_n(y) \right) \rightarrow 0$. In this context, an \textbf{invariant section} is a Borel map $v: X \rightarrow \mathcal{H}$ such that for almost every $(x,y) \in E$,  $\alpha(x,y) v(y) = v(x)$.
\begin{definition}
A probability measure preserving countable Borel equivalence relation $E$ has \textbf{property (T)} if for every Borel a.e cocycle $\alpha:E \rightarrow \mathcal{U}(\mathcal{H})$ with almost-invariant sections, there exists an invariant section.
\end{definition}

\begin{proposition}[\cite{popa_cocycle_2008}]\label{prop:propTpopa}
If $A$ is a subset of $\R^n$ of finite positive measure and $E$ denotes the probability measure preserving equivalence relation on $A$ given by restricting $E(\SL_n(\Z) \act \R^n)$ to $A$, then $E$ has property (T) if and only if $n \geq 4$.
\end{proposition} 

An measure preserving countable equivalence relation $E$ is \textbf{nowhere treeable} if its restriction to any positive measure subset is not treeable. The fact about property (T) equivalence relations relevant to us is the following:

\begin{proposition}[\cite{adams_kazhdan_1990}]\label{prop:adamsspatzier}
Suppose $(X,E)$ is an ergodic probability measure preserving countable equivalence relation satisfying Kazdhan property (T) with no conull equivalence class. Then $E$ is nowhere treeable, in particular, it is not treeable.
\end{proposition}

\begin{proposition}\label{prop:nontree4}
If $n\geq 4$, the equivalence relation $E(\SL_n(\Z) \act \R^n_\text{ti}$) is not treeable. 
\end{proposition}
\begin{proof}
Suppose $n \geq 4$. Because $A = [0,1]^n \cap \R^n_\text{ti}$ has measure $1$, the restriction of the orbit equivalence relation to $A$ has property (T) by Proposition \ref{prop:propTpopa}. Thus, by Proposition \ref{prop:adamsspatzier}, it is not treeable. Further, $A$ is a complete section for $E(\SL_n(\Z) \act \R^n_\text{ti})$. Hence, by Propositions \ref{prop:section} and \ref{prop:complexclassJKL}, $E(\SL_n(\Z) \act \R^n_\text{ti})$ is also not treeable.
\end{proof}

\subsection{Complexity theorems for Archimedean Orders}
\begin{lemma}\label{lemma:mainlem}
If the equivalence relation $E(\GL_n(\Z) \act \Ar(\Z^n))$ is hyperfinite (resp. treeable), then so is $E(\SL_n(\Z) \act \Ar(\R^n))$ on a Lebesgue-co-null set.
\end{lemma}

\begin{proof}
Suppose $E(\GL_n(\Z) \act \Ar(\Z^n))$ is hyperfinite (resp. treeable). By Proposition \ref{prop:ArOrderStructure}, there is a $\GL_n(\Z)$-equivariant bijection $\Ar(\Z^n) \rightarrow \left (\R^n\right )^\ast_\text{ti}/\R_{>0}$. This shows that 
\[E(\GL_n(\Z) \act \Ar(\Z^n)) \sim_B E(\GL_n(\Z) \act \left (\R^n\right )^\ast_\text{ti}/\R_{>0}).\]
By Proposition \ref{prop:complexclassJKL}, this implies $E(\GL_n(\Z) \act \left (\R^n\right )^\ast_\text{ti}/\R_{>0})$, is also hyperfinite (resp. treeable). The quotient map $\left (\R^n\right )^\ast_\text{ti} \rightarrow \left (\R^n\right )^\ast_\text{ti}/\R_{>0}$ is a homomorphism from $E(\GL_n(\Z) \act \left (\R^n\right )^\ast_\text{ti})$ to $E(\GL_n(\Z) \act \left (\R^n\right )^\ast_\text{ti}/\R_{>0}$. This is because the $\GL_n(\Z)$ action commutes with scalar multiplication:
\[T \cdot f = f'\; \Longrightarrow \;\forall r \in \R_{>0}, \, T \cdot (rf) = rf' \;\Longrightarrow\; T [f] = [f'].  \]
Furthermore we also show that the quotient map is class-bijective. Injectivity on classes follows from the fact that if $rI \in \GL_n(\Z)$, then $r=\pm1$. Surjectivity follows again from the fact that the action commutes with scalar multiplication. In particular, $T \cdot f \in [T \cdot f] = T \cdot [f]$. We also have that $E(\GL_n(\Z) \act \left (\R^n\right )^\ast_\text{ti}) \sim_B E(\GL_n(\Z) \act \R^{n}_\text{ti})$ by Proposition \ref{prop:bireddual}. By the properties of class-bijective homomorphism established in Proposition \ref{prop:complexclassCK}, this implies that $E(\GL_n(\Z) \act \left (\R^n\right )^\ast_\text{ti})$ is hyperfinite (resp. treeable). This implies that $E(\GL_n(\Z) \act \R^{n}_\text{ti})$ is also hyperfinite (resp. treeable). Further, we have that $E(\SL_n(\Z) \act \R^{n}_\text{ti}) \subset E(\GL_n(\Z) \act \R^{n}_\text{ti})$ is a subequivalence relation. By properties of subequivalence relations, established in Proposition \ref{prop:complexclassJKL}, this implies that $E(\SL_n(\Z) \act \R^n_\text{ti})$ is hyperfinite (resp. treeable). This is what we wanted to show.
\end{proof}
 We are now ready to prove the two main theorems. 
\begin{theorem}\label{thm:noamen3}
The equivalence relation $E(\GL_n(\Z) \act \Ar(\Z^n))$ is not hyperfinite if $n \geq 3$.
\end{theorem}
\begin{proof}
This follows from Lemma \ref{lemma:mainlem} and Proposition \ref{prop:nonamen3}.
\end{proof}

\begin{theorem}\label{thm:notree4}
The equivalence relation $E(\GL_n(\Z) \act \Ar(\Z^n))$ is not treeable if $n \geq 4$.
\end{theorem}
\begin{proof}
This follows from Lemma \ref{lemma:mainlem} and Proposition \ref{prop:nontree4}.
\end{proof}

\begin{corollary}
The equivalence relation $E(\GL_n(\Q) \act \Ar(\Q^n))$ is not hyperfinite if $n \geq 3$ and not treeable if $n \geq 4$.
\end{corollary}
\begin{proof}
As a consequence of H\"{o}lder's theorem, cited here as Theorem \ref{thm:Holder}, we get that $\Ar(\Q^n)$ is also identifiable with $(\R^n)^*_{ti} / \R_{>0}$ in a $\GL_n(\Q)$ equivariant way. Thus, $E(\GL_n(\Q) \act \Ar(\Q^n))$ contains $E(\GL_n(\Z) \act \Ar(\Z^n))$ as a subequivalence relation. By Proposition \ref{prop:complexclassJKL} a), if $E(\GL_n(\Q) \act \Ar(\Q^n))$ is hyperfinite (resp. treeable), then so is $E(\GL_n(\Z) \act \Ar(\Z^n))$. The corollary then follows from Theorems \ref{thm:noamen3} and \ref{thm:notree4}.
\end{proof}

\section{Open questions}

From the use of class-bijective homomorphisms, we do not know what is the link between the isomorphism of Archimedean orders on $\Z^n$ and the complexity of $\GL_n(\R) \act \R^n$. We thus ask the following question:
\begin{question}
Does $E(\GL_n(\Z) \act \Ar (\Z^n))$ admits a Borel reduction to $E(\GL_3(\Z) \act \R^3)$?
\end{question}

One question which remains is 
\begin{question}
Is $E(\GL_3(\Z) \act \Ar (\Z^3))$ treeable?
\end{question}

As pointed out in the introduction, our theorems apply to all finitely generated groups admitting an Archimedean order. Letting the number of generators tend to infinity, we can consider the equivalence relation of automorphsm on all finitely generated Archimedean-orderable groups on the space
\[\Ar(\Z^2) \sqcup ... \sqcup \Ar(\Z^n) \sqcup ... \]

\begin{question}
What is the complexity of automorphism on all finitely generated Archimedean-orderable groups?
\end{question}

Two other questions arise naturally:
\begin{question}
What is the Borel complexity of $E(\GL_2(\Z[\frac{1}{n}]^2) \act \Ar (\Z[\frac{1}{n}]^2))$ ?
\end{question}
It follows from H\"{o}lder's theorem that for $n!$, the relations increase to the one on $\Q^2$. If one can show hyperfiniteness for these relations, this gives a natural example to study the Borel increasing union problem. An interesting fact to note is that in the case $\Z[\frac{1}{p}]^2$, these groups are lattices in $\SL_2(\R) \times \SL_2(\Q_p)$, as opposed to $\SL_2(\Z)$ being a lattice in $\SL_2(\R)$.

Another question is whether the family of relations $E(\GL_n(\Z^n) \act \Ar (\Z^n) $ exhibits rigidity in some form. For example, following $\cite{calderoni_rotation_2023}$, we could ask
\begin{question}
For $n < m$ large enough, can $E(\GL_m(\Z^m) \act \Ar (\Z^m))$ reduce to $E(\GL_n(\Z^n) \act \Ar (\Z^n) )$? Does $E(\GL_n(\Z^n) \act \Ar (\Z^n))$ always reduce to $E(\GL_m(\Z^m) \act \Ar (\Z^m)) $?
\end{question}

\bibliography{references.bib} 

\begin{thebibliography}{CMMRS23}

\bibitem[AL91]{adams_amenability_1991}
S.~Adams and R.~Lyons.
\newblock Amenability, {Kazhdan}’s property and percolation for trees, groups
  and equivalence relations.
\newblock {\em Israel Journal of Mathematics}, 75(2):341--370, October 1991.

\bibitem[AS90]{adams_kazhdan_1990}
S.~R. Adams and R.~J. Spatzier.
\newblock Kazhdan {Groups}, {Cocycles} and {Trees}.
\newblock {\em American Journal of Mathematics}, 112(2):271--287, 1990.
\newblock Publisher: Johns Hopkins University Press.

\bibitem[Cal23]{calderoni_rotation_2023}
F.~Calderoni.
\newblock Rotation equivalence and cocycle superrigidity.
\newblock {\em Journal of the London Mathematical Society}, 107(1):189--212,
  2023.
\newblock \_eprint: https://onlinelibrary.wiley.com/doi/pdf/10.1112/jlms.12684.

\bibitem[CC23a]{calderoni_borel_2023}
F.~Calderoni and A.~Clay.
\newblock The {Borel} complexity of the space of left-orderings,
  low-dimensional topology, and dynamics, May 2023.

\bibitem[CC23b]{calderoni_condensation_2023}
F.~Calderoni and A.~Clay.
\newblock Condensation and left-orderable groups, December 2023.
\newblock arXiv:2312.04993 [math].

\bibitem[CFW81]{connes_amenable_1981}
A.~Connes, J.~Feldman, and B.~Weiss.
\newblock An amenable equivalence relation is generated by a single
  transformation.
\newblock {\em Ergodic Theory and Dynamical Systems}, 1(4):431--450, December
  1981.
\newblock Publisher: Cambridge University Press.

\bibitem[CK18]{chen_structurable_2018}
R.~Chen and A.~S. Kechris.
\newblock Structurable equivalence relations.
\newblock {\em Fundamenta Mathematicae}, 242:109--185, 2018.
\newblock Publisher: Instytut Matematyczny Polskiej Akademii Nauk.

\bibitem[CMMRS23]{calderoni2023anticlassification}
F.~{C}alderoni, D.~{M}arker, L.~{M}otto {R}os, and A.~{S}hani.
\newblock Anti-classification results for groups acting freely on the line.
\newblock {\em Advances in Mathematics}, 418:1089938, 2023.
\newblock Accessed as a preprint in 2020.

\bibitem[CR16]{clay_ordered_2016}
A.~Clay and D.~Rolfsen.
\newblock {\em Ordered {Groups} and {Topology}}, volume 176 of {\em Graduate
  {Studies} in {Mathematics}}.
\newblock American Mathematical Society, Providence, Rhode Island, November
  2016.

\bibitem[DJK94]{dougherty_structure_1994}
R.~Dougherty, S.~Jackson, and A.~S. Kechris.
\newblock The {Structure} of {Hyperfinite} {Borel} {Equivalence} {Relations}.
\newblock {\em Transactions of the American Mathematical Society},
  341(1):193--225, 1994.
\newblock Publisher: American Mathematical Society.

\bibitem[JKL02]{jackson_countable_2002}
S.~Jackson, A.~S. Kechris, and A.~Louveau.
\newblock {Countable} {Borel} {Equivalence} {Relations}.
\newblock {\em Journal of Mathematical Logic}, 02(01):1--80, May 2002.

\bibitem[Pou22]{poulin}
Antoine Poulin.
\newblock Complexity of the isomorphism of {Archimedean} orders on {Z}², 2022.
\newblock Publisher: McGill University.

\bibitem[PV08]{popa_cocycle_2008}
S.~Popa and S.~Vaes.
\newblock Cocycle and orbit superrigidity for lattices in {SL}(n,{R}) acting on
  homogeneous spaces.
\newblock In B.~Farb and D.~Fisher, editors, {\em Geometry, Rigidity, and Group
  Actions}, chapter~13, pages 419--451. University of Chicago Press, October
  2008.

\bibitem[Zim77]{zimmer_hyperfinite_1977}
R.~J. Zimmer.
\newblock Hyperfinite factors and amenable ergodic actions.
\newblock {\em Inventiones mathematicae}, 41(1):23--31, February 1977.

\bibitem[Zim78a]{zimmer_amenable_1978}
R.~J Zimmer.
\newblock Amenable ergodic group actions and an application to {Poisson}
  boundaries of random walks.
\newblock {\em Journal of Functional Analysis}, 27(3):350--372, March 1978.

\bibitem[Zim78b]{zimmer_induced_1978}
R.~J. Zimmer.
\newblock Induced and amenable ergodic actions of {Lie} groups.
\newblock {\em Annales scientifiques de l'École Normale Supérieure},
  11(3):407--428, 1978.

\bibitem[Zim81a]{zimmer_cohomology_1981}
R.~J. Zimmer.
\newblock On the {Cohomology} of {Ergodic} {Actions} of {Semisimple} {Lie}
  {Groups} and {Discrete} {Subgroups}.
\newblock {\em American Journal of Mathematics}, 103(5):937--951, 1981.
\newblock Publisher: Johns Hopkins University Press.

\bibitem[Zim81b]{zimmer_orbit_1981}
R.~J. Zimmer.
\newblock Orbit equivalence and rigidity of ergodic actions of {Lie} groups.
\newblock {\em Ergodic Theory and Dynamical Systems}, 1(2):237--253, June 1981.
\newblock Publisher: Cambridge University Press.

\bibitem[Zim84]{zimmer_ergodic_1984}
R.~J. Zimmer.
\newblock {\em Ergodic {Theory} and {Semisimple} {Groups}}.
\newblock Birkhäuser, Boston, MA, 1984.

\end{thebibliography}
\bibliographystyle{alpha}
\end{document}